\documentclass[draft,reqno,b5paper]{amsart}
\usepackage{amsmath}
\usepackage{amssymb}

\usepackage{amssymb, amsmath}
\usepackage{url}

\theoremstyle{plain}
\newtheorem{theorem}{Theorem}[section]
\newtheorem{prop}[theorem]{Proposition}

\theoremstyle{definition}
\newtheorem{definition}[theorem]{Definition}

\theoremstyle{remark}

\numberwithin{equation}{section}

\let\abs=\envert

\newcommand{\sgn}{\textrm{sgn}}
\newcommand{\Fnhat}{\hat{F_n}}
\newcommand{\That}{\hat T}
\newcommand{\Fhat}{\hat F}
\newcommand{\fhat}{\hat f}
\newcommand{\ghat}{\hat g}
\newcommand{\Ghat}{\hat G}
\newcommand{\phihat}{\hat \phi}
\newcommand{\Ltwop}{\Lany{2}}
\newcommand{\Lonep}{\Lany{1}}
\newcommand{\Lpp}{L{\!}'^{\,p}}
\newcommand{\Lnp}{L{\!}^{(n),p}}
\newcommand{\Inq}{I^{n,q}}

\newcommand{\Br}{{\mathcal B}_r}
\newcommand{\Bc}{{\mathcal B}_c}
\newcommand{\acn}{{\mathcal A}^n_c}
\newcommand{\arn}{{\mathcal A}^n_r}

\newcommand{\supp}{\textrm{supp}}

\newcommand{\alexc}{{\mathcal A}_c}
\newcommand{\balexr}{{\mathcal B}_r}
\newcommand{\balexc}{{\mathcal B}_c}
\newcommand{\Rbar}{\overline{\RR}}

\newcommand{\bv}{{\mathcal BV}}
\newcommand{\intinf}{\int^\infty_{-\infty}}

\newcommand{\NN}{{\mathbb N}}
\newcommand{\RR}{{\mathbb R}}

\newcommand{\Sc}{{\mathcal S}}
\newcommand{\fn}{\!:\!}

\DeclareMathOperator{\esssup}{ess\,sup}

\providecommand{\abs}[1]{\lvert#1\rvert}
\providecommand{\norm}[1]{\lVert#1\rVert}
\providecommand{\Lany}[1]{L{\!}'^{\,#1}}

\begin{document}
\date{Preprint August 17, 2012.  To appear in \textit{Mathematica Slovaca}.}


\title{The $L^p$ primitive integral}
\author[Erik Talvila]{Erik Talvila*}

\newcommand{\acr}{\newline\indent}

\address{\llap{*\,}Department of Mathematics \& Statistics\acr
                   University of the Fraser Valley\acr
                   Abbotsford, BC V2S 7M8\acr
                   CANADA}
\email{Erik.Talvila@ufv.ca}

\begin{abstract}
For each $1\leq p<\infty$ a space of integrable Schwartz distributions,
$L{\!}'^{\,p}$,
is defined by taking the distributional derivative of all functions
in $L^p$.  Here,  $L^p$ is with respect to 
Lebesgue measure on the real line.  If $f\in L{\!}'^{\,p}$ such that $f$ is
the distributional derivative of $F\in L^p$ then
the integral is defined as
$\int^\infty_{-\infty} fG=-\int^\infty_{-\infty} F(x)g(x)\,dx$, where $g\in L^q$, $G(x)=
\int_0^x g(t)\,dt$ and $1/p+1/q=1$.
A norm is $\lVert f\rVert'_p=\lVert F\rVert_p$.
The spaces $L{\!}'^{\,p}$ and $L^p$ are isometrically isomorphic.
Distributions in $L{\!}'^{\,p}$ share many properties with functions in $L^p$.
Hence, $L{\!}'^{\,p}$ is reflexive, its dual space is identified with $L^q$,
there is a type of H\"older inequality, continuity
in norm, convergence theorems, 
Gateaux derivative.  It is a Banach lattice and abstract $L$-space.
Convolutions and Fourier transforms are defined.
Convolution with the Poisson kernel is
well-defined and provides a solution to the half plane Dirichlet problem, 
boundary values being taken on in the new norm.
A product is defined that
makes $L{\!}'^{\,1}$
into a Banach algebra isometrically isomorphic to the 
convolution algebra
on $L^1$.
Spaces of higher order
derivatives of $L^p$ functions are defined.  These are also Banach spaces
isometrically isomorphic to $L^p$.
\end{abstract}

\maketitle

\section{Introduction}\label{sectionintroduction}
One way of defining an integral is through properties of its
primitive.  This is a function whose derivative is in some 
sense equal to the integrand.  For example, if $f\fn[a,b]\to\RR$
then $f\in L^1$ if and only if there is an absolutely continuous
function $F$ such that $F'(x)=f(x)$ for almost all $x\in(a,b)$.
This then provides a descriptive definition of the Lebesgue integral in
terms of the primitive $F$ and the fundamental theorem of calculus
formula $\int_a^bf(x)\,dx=F(b)-F(a)$.  
This same approach can be used to
define Henstock--Kurzweil and wide Denjoy integrals.  See \cite{celidze}
for the relevant spaces of primitives.  

In this paper we define
integrals of tempered distributions by taking $L^p$ as the space of
primitives for $1\leq p<\infty$.  Such functions need not have pointwise
derivatives so the distributional derivative is used.  The distributions
integrable in this sense are the weak derivative of $L^p$ functions but
have many properties similar to
$L^p$
functions.  This approach was
followed in \cite{talviladenjoy} with the continuous primitive integral.
The primitives were functions continuous on the extended real line.  The
space of distributions integrable in this sense is a Banach space under
the Alexiewicz norm, isometrically
isomorphic to the space of primitives with the uniform norm.
Primitives were taken to be regulated functions in 
\cite{talvilaregulated}. 
A function on the real line is
regulated if it has a left limit and a right limit at each point.  This
again led to a Banach space of distributions that was isometrically
isomorphic to the space of primitives with the uniform norm.  The
space of distributions that have a
continuous primitive integral is the completion of $L^1$ and the space
of Henstock--Kurzweil integrable functions in the Alexiewicz norm.
The regulated primitive integral provides the completion of the signed
Radon measures in the Alexiewicz norm.  In the current paper we again
define a Banach space of distributions, only now it is isometrically
isomorphic to an $L^p$ space.

The outline of the paper is as follows.

Section~\ref{sectionnotation} provides some notation for integrals
and distributions.  Then in 
Section~\ref{sectionLp} we define the space $\Lpp$ to be the
set of distributions that are the
distributional derivative of $L^p$ functions
on the real line.  This is our space of integrable distributions.
We define $I^q$ to be the absolutely continuous functions that are
the indefinite integral of functions in $L^q$.  See 
Definition~\ref{defnprimitive}.  Such functions are
multipliers and the integral is defined as $\intinf fG:=-\intinf
F(x)G'(x)\,dx$ where $f\in\Lpp$ such that $f=F'$ for $F\in L^p$
and $G(x)=\int_0^xg(t)\,dt$ for a function $g\in L^q$ (Definition~
\ref{defnintegral}).  Here,
$1\leq p<\infty$ and $1/p+1/q=1$.  We are therefore defining $\intinf fG$
in terms of the Lebesgue integral $\intinf F(x)G'(x)\,dx$
with respect to Lebesgue measure.  Primitives are unique
since there are no constant functions in $L^p$ (Theorem~\ref{theoremunique}).  
A norm on $\Lpp$ is then
$\norm{f}'_p=\norm{F}_p$.  This makes $\Lpp$ into a Banach space 
isometrically isomorphic to $L^p$.  Many properties of $L^p$ are
then inherited by $\Lpp$.  The remaining theorems and propositions in
this section show
that $\Lpp$ is separable,
the dual space is isometrically isomorphic to $L^q$,
the unit ball is uniformly convex, $\Lpp$ is reflexive, it is invariant
under translations, there is continuity in norm, the Schwartz space of
rapidly decreasing $C^\infty$ functions is dense, there are versions
of the H\"older and Hanner inequalities.  A Gateaux derivative is computed.
We prove a norm convergence
theorem and define an equivalent norm.  
Under pointwise operations $L^p$ is a Banach lattice.
If $f,g\in\Lpp$ with primitives $F,G\in L^p$, define
$f\preceq g$ whenever $F(x)\leq G(x)$ for almost all $x\in\RR$.
This makes $\Lpp$ into a Banach lattice that is lattice isomorphic to $L^p$.
It is then an abstract $L$-space in the sense of Kakutani.  Under this
ordering there is a version of the dominated convergence theorem.  At
the end of this section we show how to approximate the integral by a
sequence of derivatives of step functions.

In Section~\ref{sectionexamples} various examples are given.  
Proposition~\ref{propintcondition} gives an integral condition with a 
power growth weight that ensures a function is in
$\Lpp$.  Examples are given of functions or distributions in $\Lpp$
that are not in
$L^1_{loc}$, not in any $L^p$ space, are the difference of translations
of the Dirac distribution, have conditionally convergent integrals,
have principal value integrals,
have primitives whose pointwise derivative vanishes almost everywhere
or exists nowhere.

Convolutions are defined as $\ast\fn \Lpp\times I^q\to L^\infty$
for $p$ and $q$ conjugate.
Various properties are proved in Theorem~\ref{theoremconvinfty}.  In this
case there are many results similar to convolutions defined on
$L^p\times L^q$, such as uniform continuity.  
Convolutions are defined in Definition~\ref{defnconvp} as 
$\ast\fn \Lpp\times L^q\to \Lany{r}$ and 
$\ast\fn L^p\times \Lany{q} \to \Lany{r}$
using
a sequence in $L^q\cap I^q$ that converges to a given
function in $L^q$.  Here, $p,q,r\in[1,\infty)$ such that $1/p+1/q=1+1/r$.
Properties of the convolution that mirror properties
of convolutions in $L^p\times L^q$ are proved in Theorem~\ref{theoremconvLp}.
A different type of product is defined in Theorem~\ref{banachalgebra} that
makes $\Lany{1}$ into a Banach algebra isometrically isomorphic to the convolution algebra
on $L^1$.

Since $\Lonep$ is isometrically isomorphic to $L^1$, Fourier transforms
can be defined directly using the usual integral definition. If
$f\in\Lonep$ then its Fourier transform is given by the integral
$\fhat(s)=\intinf f(t)e^{-ist}\,dt=is\Fhat(s)$, where $F\in L^1$ is the
primitive of $f$.  
It is shown
that this agrees with the definition for tempered distributions.
Fourier transforms of distributions in $\Lonep$ are continuous functions
and the Riemann--Lebesgue lemma takes the form $\fhat(s)=o(s)$ as
$|s|\to\infty$.  Many of the usual properties of $L^1$ Fourier transforms
continue to hold in $\Lonep$.  See Theorem~\ref{theoremFourier}.

In Section~\ref{sectionL2} some special properties of $\Ltwop$ are
considered.
The space $\Ltwop$ is isometrically isomorphic to $L^2$ so it is a
Hilbert space.  The inner product is $(f,g)=(F,G)=\intinf F(x)G(x)\,dx$,
where $f,g\in \Ltwop$ with respective primitives $F,G\in L^2$.  The
Fourier transform is defined from the $L^2$ Fourier transform of the
primitive.  

In Section~\ref{sectionhigher} spaces of distributions are constructed
by taking the $n$th distributional derivative of $L^p$ functions.  Each
such space is then a separable Banach space, isometrically isomorphic
to $L^p$.  Most of the results for $\Lpp$ continue to hold in these
spaces.

In Section~\ref{sectionpoisson} it is shown how the half plane Poisson integral
can be defined for distributions that are the $n$th derivative of an
$L^p$ function.  There are direct analogues of the usual $L^p$ results,
such as boundary values being taken on in the $\norm{\cdot}'_p$ norm.

In Section~\ref{sectionRn} we sketch out how these integrals
can be defined in $\RR^n$.

\section{Notation}\label{sectionnotation}
All statements regarding measures will be with respect to Lebesgue
measure, denoted $\lambda$. 
For $1\leq p<\infty$,
the Lebesgue space on the real line is
$L^p$, which consists of the measurable functions $f\fn\RR\to\RR$
such that $\intinf |f(x)|^p\,dx<\infty$.  To distinguish between other
types of integrals introduced later, Lebesgue integrals will always
explicitly show the
integration variable and differential as above.
The $L^p$ spaces have norm
$\norm{f}_p=(\intinf\! |f(x)|^p\,dx)^{1/p}$.
And, $L^\infty$ is the set of bounded measurable functions with
norm $\norm{f}_\infty=\esssup_{x\in\RR}|f(x)|$.
For $1\leq p\leq\infty$ each $L^p$ is a Banach space.
If $1< p<\infty$ then its conjugate exponent is $q\in\RR$
such that
$p^{-1}+q^{-1}=1$.  For $p=1$, $q=\infty$.
The locally integrable functions are
$L^p_{loc}$ and a measurable function $f\in L^p_{loc}$ if $f\chi_{[a,b]}\in
L^p$ for each compact interval $[a,b]$.  The set of absolutely continuous 
functions on the real
line is denoted
$AC(\RR)$ and consists of the functions $F\fn\RR\to\RR$ such that $F$ is
absolutely continuous on each compact interval in $\RR$.  And,
$F\in AC(\RR)$ if and only if there is $f\in L^1_{loc}$ such that
$F(x)=F(0)+\int_0^xf(t)\,dt$.

The extended real line is $\Rbar=[-\infty,\infty]$ and $C(\Rbar)$
denotes the real-valued functions that are continuous at each
point of $\RR$ and have real limits at $-\infty$ and at $\infty$.
Define $AC(\Rbar)=AC(\RR)\cap\bv$ where $\bv$ are the functions of
bounded variation on the real line.  Then $f\in L^1$ if and only if there is
$F\in AC(\Rbar)$ such that $f(x)=F'(x)$ for almost all $x\in\RR$.

The Schwartz space, $\Sc$, of rapidly decreasing smooth functions, consists of 
the functions $\phi\in C^\infty
(\RR)$ such that for all integers $m,n\geq 0$ we have $x^m\phi^{(n)}(x)\to 0$
as $|x|\to\infty$.  Elements of $\Sc$ will be termed 
test functions.  Sequence $(\phi_j)\subset\Sc$ is said to converge to
$\phi\in\Sc$ if for all integers $m,n\geq 0$, 
$\sup_{x\in\RR}|x|^m|\phi^{(n)}_j(x)-\phi^{(n)}(x)|\to0$ as $j\to\infty$.
The (tempered) distributions are then the continuous linear
functionals on $\Sc$.  This dual space is denoted $\Sc'$.  
If $T\in\Sc'$ then $T\fn\Sc'\to\RR$
and we write $\langle T,\phi\rangle\in\RR$ for $\phi\in\Sc$.  If
$\phi_j\to\phi$ in $\Sc$ then $\langle T,\phi_j\rangle\to\langle T,\phi\rangle$
in $\RR$.  And, for all $a_1, a_2\in\RR$ and all $\phi,\psi\in\Sc$, $\langle T,
a_1\phi+a_2\psi\rangle =a_1\langle T,\phi\rangle+a_2\langle T,\psi\rangle$.  If
$f\in L^p_{loc}$  for some $1\leq p\leq\infty$ then $\langle T_f,\phi\rangle=\int_{-\infty}^\infty
f(x)\phi(x)\,dx$ defines a distribution.  In such case, $T_f$ is
called regular and we
will often ignore the distinction between $f$ and $T_f$. The  
differentiation  formula $\langle D^nT,\phi\rangle=\langle T^{(n)},
\phi\rangle=(-1)^n\langle T,\phi^{(n)}\rangle$ ensures that all 
distributions have 
derivatives of all  orders which are themselves distributions.  
This is
known as the distributional or weak derivative.
We  will  usually denote distributional
derivatives by $D^nF$, $F^{(n)}$ or $F'$ and  pointwise derivatives by $F^{(n)}(t)$
or $F'(t)$.
For $T\in\Sc'$ and $t\in\RR$  the translation $\tau_t$ is defined
by $\langle\tau_tT, \phi\rangle=\langle T, \tau_{-t}\phi\rangle$
where $\tau_t\phi(x)=\phi(x-t)$ for
$\phi\in\Sc$.  
The Heaviside step function
is $H=\chi_{(0,\infty)}$.  The Dirac distribution is $\delta=H'$.
The action of $\delta$ on test function $\phi$ is $\langle\delta,\phi\rangle
=\phi(0)$.  See \cite{folland,friedlanderjoshi,zemanian} 
for more on
distributions.

Laurent Schwartz introduced the notion of integrable distribution,
see \cite[p.~199-203]{schwartz} and \cite{barros-neto}.
Define the test function spaces as
${\mathcal D}_{L^p}(\RR)=\{\phi\in C^\infty(\RR)\mid \phi^{(m)}\in L^p(\RR)
\text{ for each } m\geq 0\}$ for $1\leq p<\infty$.  For $p=\infty$, define
${\mathcal B}={\mathcal D}_{L^\infty}(\RR)=\{\phi\in C^\infty(\RR)\mid 
\norm{\phi}_\infty<\infty\}$ and
${\stackrel.{\mathcal B}}={\stackrel.{\mathcal D}}_{L_\infty}(\RR)
=\{\phi\in {\mathcal D}_{L^\infty}(\RR)\mid
\lim_{|x|\to\infty}\phi(x)=0\}$.
For $1\leq p<\infty$, a sequence $(\phi_n)\subset {\mathcal D}_{L^p}(\RR)$
(or in ${\stackrel.{\mathcal D}}_{L_\infty}(\RR)$)
converges to $\phi\in{\mathcal D}_{L^p}(\RR)$ 
(or $\phi\in{\stackrel.{\mathcal D}}_{L^\infty}(\RR)$), if
$\lim_{n\to\infty}\norm{\phi^{(m)}_n-\phi^{(m)}}_p=0$
for each $m\geq 0$.
The integrable distributions are then ${\mathcal D}'_{L^p}(\RR)$ which
is the dual of ${\mathcal D}_{L^q}(\RR)$ ($1<p<\infty$, $1/p+1/q=1$) and
${\mathcal D}'_{L^1}(\RR)$ which is the dual of ${\stackrel.{\mathcal B}}$.
Schwartz's main structure theorem \cite[p.~201]{schwartz} is that if
$1\leq p\leq\infty$ and $T$ is a distribution then 
$T\in{\mathcal D}'_{L^p}(\RR)$ if and only if $T=\sum_{n=0}^m F_n^{(n)}$ for
some $F_n\in L^p(\RR)$ and some $m\geq 0$.  If this expansion holds then
the functions $F_n$ can be taken to be
bounded and continuous.

Our theory differs in that we take $T=F^{(m)}$ for some $F\in L^p(\RR)$ and
some $m\geq 1$.  This is a restricted form of Schwartz's definition but it
has the advantage that the resulting space of distributions is a Banach
space isometrically isomorphic to $L^p(\RR)$.  This provides a class of 
distributions that behave in many ways like $L^p$ functions.

\section{The $L^p$ primitive integral}\label{sectionLp}
In this section we define Banach spaces $\Lpp$ and $I^q$ that
are isometrically isomorphic to $L^p$ and $L^q$, respectively.
The first serves as a space of integrable distributions and the
second as a space of multipliers.  The distributional
derivative provides a linear isometry
between $L^p$ and $\Lpp$ and many properties of $L^p$ are inherited
by $\Lpp$.  Hence, $\Lpp$ is a separable Banach space, reflexive with
dual space isometrically isomorphic to $L^q$ for $1/p+1/q=1$.  There is
a H\"older inequality and we prove a convergence theorem.  The pointwise
ordering on $L^p$ is inherited by $\Lpp$ so that it is a Banach lattice
and abstract $L$-space.  A version of the dominated convergence theorem
based on this ordering is given.  At the end of this section, the
integral is also defined in terms of the limit of a
sequence of derivatives of step functions.

\begin{definition}\label{defnprimitive}
Let $1\leq p\leq\infty$.
(a) Define $\Lpp=\{f\in\Sc'\mid f=F'
\text{ for } F\in L^p\}$. (b) Define $I^p=\{G\fn\RR\to\RR\mid
G(x)=\int_0^xg(t)\,dt \text{ for some } g\in L^p\}$.
\end{definition}

\begin{theorem}\label{theoremunique}
(a) Let $1\leq p<\infty$.
If $f\in\Lpp$ there is a unique function $F\in L^p$ such that
$f=F'$.
(b) Let $1\leq p\leq \infty$.
If $G\in I^p$ there is a unique function $g\in L^p$ such that
$G(x)=\int_0^xg(t)\,dt$.
\end{theorem}

\begin{proof} (a) If $f\in\Lpp$ and $f=F_1'=F_2'$ then let $F=F_1-F_2$.
Hence, $F\in L^p$ and $F'=0$ in $\Sc'$.  It follows that $F$ is a
constant distribution \cite[\S2.4]{friedlanderjoshi}.  The only constant 
distribution in $L^p$ is $0$. (b)  If there are $g_1,g_2\in L^p$ such
that $G(x)=\int_0^xg_1(t)\,dt =\int_0^xg_2(t)\,dt$ for all $x\in\RR$
then $g_1=g_2$ almost everywhere.\end{proof}

This uniqueness is within the equivalence class structure on
$L^p$.  Two functions in $L^p$ are equivalent if they are equal almost
everywhere.  We always consider $L^p$ as a disjoint union of these
equivalence classes. The
unique function $F$ in Theorem~\ref{theoremunique} is called the
primitive of $f$.  If $f\in \Lpp$ and $F$ is its primitive in 
$L^p$ then $\langle f,\phi\rangle=\langle F',\phi\rangle=
-\langle F,\phi'\rangle=-\intinf F(x)\phi'(x)\,dx$ for all $\phi\in\Sc$.
Since $\phi'$ is in each $L^q$ for all $1\leq q\leq \infty$, this last
integral exists by the H\"older inequality.

Now we can show that $\Lpp$ is a Banach space isometrically isomorphic
to $L^p$.

\begin{theorem}\label{theoremLp}
Let $1\leq p<\infty$.
Let $f,f_1,f_2\in\Lpp$.  Let their respective primitives in $L^p$ be
$F, F_1, F_2$.  Let $a_1,a_2\in\RR$.  Define $a_1f_2+a_2f_2=(a_1F_1+a_2F_2)'$.
Define $\norm{f}'_p=\norm{F}_p$.  Then $\Lpp$ has the following properties.  
(a) It is a Banach space with norm $\norm{\cdot}'_p$, isometrically
isomorphic to $L^p$.  (b) It is separable.
(c) Its dual space is isometrically isomorphic to $I^q$ where
$q$ is conjugate to $p$. 
\end{theorem}

\begin{proof} Define $D\fn L^p\to\Lpp$ by $D(F)=F'$.  By 
Theorem~\ref{theoremunique}, $D$ is injective.  And $D$ is
surjective by the definition of
$\Lpp$.  It then follows that $L^p$ and $\Lpp$ are isometrically
isomorphic.  The other properties then follow from the corresponding
properties in $L^p$.\end{proof}

The inverse $D^{-1}\fn\Lpp\to L^p$ can be formally computed as follows.
Let $x\in\RR$ and define $G_x\in L^\infty$ by $G_x=\chi_{(-\infty,x)}$.
Let $\delta_x$ be Dirac measure supported at $x$.  Let $f=F'\in\Lpp$.
Then $\intinf fG_x=-\intinf F(t)G_x'(t)\,dt=\intinf F(t)\tau_x\delta(t)\,dt
=\intinf F(t)\,d\delta_x(t)=\langle \tau_x\delta, F\rangle =F(x)$ for
almost all $x\in\RR$. This a purely formal calculation since
$\intinf fG_x$ has not been defined and $G_x$ is not in $L^p$ 
for any $1\leq p<\infty$.  However, let $n\in\NN$ such that $n>-x$ and define
$$
G_{n,x}(t)=\left\{\!\!\begin{array}{cl}
0, & t\leq -2n\\
(t+2n)/n, & -2n\leq t\leq -n\\
1, & -n\leq t\leq x\\
-n(t-x-1/n), & x\leq t\leq x+1/n\\
0, & t\geq x+1/n.
\end{array}
\right.
$$
Then $G_{n,x}\in I^q$ for each $1\leq q\leq\infty$.  Define
$F_n(x)=\intinf fG_{n,x}$.  Then 
$F_n(x)=-\intinf F(t)G'_{n,x}(t)\,dt=A_n+B_n(x)$ where
$A_n=-n^{-1}\int_{-2n}^{-n}F(t)\,dt$ and $B_n(x)=n\int_{x}^{x+1/n}F(t)\,dt$.  
By the H\"older inequality $A_n\to 0$ as $n\to\infty$.  By the Lebesgue
differentiation theorem $B_n(x)\to F(x)$ for almost every $x\in\RR$.
Hence, $\lim_{n\to\infty}F_n(x)=F(x)$ for almost every $x\in\RR$ as 
$n\to\infty$.  This then gives a means of computing $D^{-1}$.

The case $p=\infty$ is rather different as the
dual space of $L^\infty$ is not one of the $L^p$ spaces.  See 
\cite[IV.9 Example~5]{yosida}.
We postpone
this case for consideration elsewhere.

The space $I^q$ is isometrically isomorphic to the dual
space of $\Lpp$.
\begin{theorem}
Let $1\leq q\leq\infty$. (a) $I^q\subset AC(\RR)$. (b) 
For $G\in I^q$ define $\norm{G}_{I,q}=
\norm{G'}_q$ then under usual pointwise operations $I^q$ is a
Banach space isometrically isomorphic to $L^q$.
\end{theorem}
This is an immediate consequence of the fundamental theorem of calculus
and Theorem~\ref{theoremunique}(b).
The pointwise derivative operator $D\fn I^q\to L^q$
defines a linear isometry.  Note that if
$G\in I^q$ for some $1<q\leq\infty$ then $G$ need not be bounded.
The space $I^\infty$ is the same as the Lipschitz functions that
vanish at $0$.  

The Sobolev space $W^{1,p}(\RR)$ consists of the absolutely
continuous $L^p$ functions
whose distributional derivative is also in $L^p$.  It is a Banach
space with norm $\norm{f}_{1,p}=\norm{f}_p+\norm{f'}_p$.  See
\cite{yosida}.  Clearly, $W^{1,p}(\RR)$ is a subspace of $I^p$.  It is not
a closed subspace.  For example, let 
$g_n=\beta_n\chi_{(0,\alpha_n)}-\beta_n\chi_{(\alpha_n,2\alpha_n)}$,
where $\alpha_n,\beta_n>0$ for each $n\in\NN$.  Suppose $1\leq p<\infty$.
Then $g_n\in L^p$ and $\norm{g_n}_p=2^{1/p}\alpha_n^{1/p}\beta_n$.
Let $G_n(x)=\int_0^xg_n(t)\,dt$.  Then $G_n\in AC(\RR)$ and $G_n(x)=0$
for $|x|\geq 2\alpha_n$ so $G_n\in L^p$.  And, $\norm{G_n}_p=
2^{1/p}(p+1)^{-1/p}\alpha_n^{1+1/p}\beta_n$.  For $p=\infty$ we
have $\norm{g_n}_\infty=\beta_n$ and $\norm{G_n}_\infty=\alpha_n\beta_n$.
Now let $\alpha_n=n^2$ and
$\beta_n=n^{-(1+2/p)}$.  Then $\norm{g_n}_p\to 0$ as $n\to \infty$ so
$G_n\to 0$ in $I^p$, while $\norm{G_n}_p\to\infty$ so $(G_n)$ does not
converge in $W^{1,p}(\RR)$.  For each $1\leq p\leq\infty$ then, $W^{1,p}(\RR)$ is not closed in $I^p$.

Now we can define an integral on $\Lpp$. 
\begin{definition}\label{defnintegral}
Let $1\leq p<\infty$ and let $q$ be its conjugate.
Let $f\in\Lpp$ and let $G\in I^q$.  The integral of $fG$ is
$\intinf fG=-\intinf F(x)G'(x)\,dx$ where $F\in L^p$ is the
primitive of $f$.  For each $a\in\RR$ define $\intinf fa=0$.
\end{definition}
This defines a bilinear product $\Lpp\times I^q\to
L^1$ with $(f,G)\mapsto -FG'$.
Some mathematicians may wish to refer to $\intinf fG$ as merely
a linear functional but we like the term integral.
The Lebesgue integral $\intinf F(x)G'(x)\,dx$ exists by the
H\"older inequality.  This leads to a version of the H\"older
inequality in $\Lpp$ and $I^q$.
\begin{theorem}[H\"older inequality]\label{theoremholder}
Let $1\leq p<\infty$ and let $q$ be its conjugate.
Let $f\in\Lpp$ with primitive $F\in L^p$ and let $G\in I^q$.   Then
$$\left|\intinf fG\right|=\left|\intinf F(x)G'(x)\,dx\right|
\leq \norm{F}_p\norm{G'}_q= \norm{f}'_p\norm{G}_{I,q}.$$
\end{theorem}

A consequence of the H\"older inequality is the following convergence
theorem.
\begin{theorem}\label{theoremconvergence}
Let $1\leq p<\infty$ with conjugate $q$.
Suppose $f, f_n\in\Lpp$ and $G,G_n\in I^q$ for each $n\in\NN$.
If $\norm{f_n-f}'_p\to 0$ and $\norm{G_n-G}_{I,q}\to 0$ then
$\intinf f_nG_n\to\intinf fG$.
\end{theorem}
The proof follows from the equality $f_nG_n-fG=(f_n-f)G_n+f(G_n-G)$,
linearity of the distributional derivative and the fact that 
$\norm{G_n}_{I,q}$
is bounded.

The H\"older inequality can also be used to demonstrate an
equivalent norm.  The corresponding result for $L^p$
appears in \cite[Proposition~6.13]{folland}.
\begin{prop}\label{eqnorms}
Let $f\in\Lpp$ with primitive $F\in L^p$ and conjugate $q$.
Define $\norm{f}''_p=\sup_G\intinf fG$ where the supremum is
taken over all $G\in I^q$ such that $\norm{G}_{I,q}\leq 1$.
Then $\norm{f}'_p=\norm{f}''_p$.
\end{prop}

\begin{proof} By the H\"older inequality, 
$\norm{f}''_p\leq\norm{f}'_p$.  Without loss of generality $f\not=0$.  Let
$g(x)=\textrm{sgn}(F(x))|F(x)|^{p-1}\norm{F}_p^{1-p}$.  Then
$|g(x)|^q=|F(x)|^p\norm{F}_p^{-p}$ so $g\in L^q$ and $\norm{g}_q=1$.
Hence, $g\in L^1_{loc}$ so we can define $G(x)=\int_0^xg(t)\,dt$.
Then
$$
\norm{f}''_p  \geq  \left|\intinf F(x)g(x)\,dx\right|=
\frac{1}{\norm{F}_p^{p-1}}\intinf |F(x)|^p\,dx
  =  \norm{F}_p=\norm{f}_p'.
$$
This shows that $\norm{f}'_p=\norm{f}''_p$.\end{proof} 

Step functions can be used to
give an alternate definition of the integral.  If $I_n=(x_{n},y_n)$
are finite, disjoint intervals for $1\leq n\leq N$ then a step function is
$\sigma=\sum_{1}^Na_n\chi_{I_n}$, where $a_n\in\RR$.  The  integral is
$\intinf\sigma(x)\,dx=\sum_1^N a_n[y_n-x_n]$.  And, if $g\in L^q$ with
$G(x)=\int_0^xg(t)\,dt$ then $\intinf \sigma(x)g(x)\,dx
=\sum_1^Na_n[G(y_n)-G(x_n)]$.  The step functions are dense in $L^p$ so
for each $F\in L^p$ there is a sequence of step functions $(\sigma_n)$
such that $\norm{F-\sigma_n}_p\to 0$.  Now suppose $f=F'\in\Lpp$.
Define $s=\sigma'=\sum_1^Na_n[\tau_{x_n}\delta-\tau_{y_n}\delta]$
and $s_n=D(\sigma_n)$.
Then $|\intinf(fG-s_nG)|=|\intinf(Fg-\sigma_ng)|\leq
\norm{F-\sigma_n}_p
\norm{g}_q$.  Using Proposition~\ref{eqnorms} and taking the supremum over 
all $G\in I^q$ such that
$\norm{G}_{I,q}\leq 1$ shows $\norm{f-s_n}'_p\to 0$.   
Hence, we can approximate distributions in 
$\Lpp$ by differences of Dirac distributions as with $s$.  The
integral of $sG$ is defined $\intinf sG=-\intinf \sigma(x)g(x)\,dx
=-\sum_1^Na_n[G(y_n)-G(x_n)]$.  This gives an alternative definition
of the integral.

If $G\in I^q$ then there is $g\in L^q$ such that $G(x)=\int_0^xg(t)\,dt$.
This imposes the arbitrary condition $G(0)=0$.
If instead, we take
$G_a(x)=\int_a^xg(t)\,dt$ for some fixed $a\in\RR$ then
$G$ and $G_a$ differ by a constant.  This does not affect the
integral in Definition~\ref{defnintegral} since it only depends
on the derivative of $G$.

If $F$ and $G$ are absolutely continuous functions then the integration
by parts formula is
$$
\intinf F'(x)G(x)\,dx=\lim_{x\to\infty}F(x)G(x)-\lim_{x\to-\infty}F(x)G(x)
-\intinf F(x)G'(x)\,dx
$$
provided these limits exist.  When $FG$ vanishes at $\pm\infty$ the
integral in Definition~\ref{defnintegral} agrees with the Lebesgue
integral.  Similarly, it agrees with the Henstock--Kurzweil and
wide Denjoy integrals.  See \cite{celidze} for the relevant
spaces of primitives for these integrals.  These limit terms are omitted in
Definition~\ref{defnintegral}, as they are in the definition of
the distributional derivative.  For $F\in L^p$ and $G\in I^q$ the
product $FG$ vanishes in the following weak sense 
as $|x|\to\infty$.  
\begin{prop}\label{proplimitFG}
Let $F\in L^p$ and $G\in I^q$ where $1\leq p<\infty$ and $q$ is conjugate
to $p$.  For $0<M<N$ and $\epsilon>0$ define
$E_{(M,N),\epsilon}=\{x\in(M,N)\mid |F(x)G(x)|>\epsilon\}$. 
Then 
\begin{eqnarray}
\frac{\lambda(E_{(M,N),\epsilon})}{N-M} & \leq & \left\{\!\!\begin{array}{cl}
\frac{\norm{F\chi_{(M,N)}}_p
\norm{G'}_q\left(N^q-M^q\right)^{1/q}}{\epsilon \,q^{1/q}(N-M)}, & 1<p<\infty\\
\label{MNlimit}
\frac{\norm{F\chi_{(M,N)}}_1
\norm{G'}_\infty N}{\epsilon(N-M)}, &  p=1.
\end{array}
\right.
\end{eqnarray}
\end{prop}
\begin{proof}
The H\"older inequality gives $\norm{FG\chi_{(M,N)}}_1\leq\norm
{F\chi_{(M,N)}}_p\norm{G\chi_{(M,N)}}_q$.  For $1<p<\infty$ use 
Jensen's inequality,
for example, \cite[p.~109]{folland}, to get
$$
\norm{G\chi_{(M,N)}}_q^q  =  \int_M^N\left|\int_0^x G'(t)\frac{dt}{x}
\right|^q x^q\,dx
  \leq  \norm{G'}_q^q(N^q-M^q)/q.
$$
For $p=1$ we have $\norm{G\chi_{(M,N)}}_\infty=\sup_{x\in(M,N)}|\int_0^x
G'(t)\,dt|\leq \norm{G'}_\infty N$.
As in the proof of the Chebyshev inequality \cite[6.17]{folland} we have
$\norm{FG\chi_{(M,N)}}_1\geq\epsilon\lambda(E_{(M,N),\epsilon})$.  The
result now follows. \end{proof}

For $1\leq p<\infty$,
let $M,N\to\infty$ in \eqref{MNlimit} such that
$M\leq \delta N$ for some $0<\delta<1$.  Since $F\in L^p$ we have 
$\norm{F\chi_{(M,N)}}_p\to 0$.  Hence, the measure of $E_{(M,N),\epsilon}$
relative to the interval $(M,N)$ tends to $0$ in this limit.  In this
weak sense, $F(x)G(x)\to 0$ as $|x|\to \infty$.

Due to the isometry,
properties of $L^p$ that depend only on the norm carry over
to $\Lpp$.  Here we list a few such properties, each with a 
reference to the $L^p$ result.  
Proofs of the $\Lpp$ result follow from the $L^p$ result using the
fact that the 
distributional derivative provides a linear isometry
and isomorphism between $\Lpp$ and $L^p$.
\begin{theorem}\label{theoremLpresults}
(a)  For $1\leq p<\infty$, $\Sc$ is dense in $\Lpp$.
(b) $\Lpp$ is uniformly convex for $1<p<\infty$.
Reflexivity follows by Milman's theorem.
(c) The unit ball of $\Lpp$ is strictly convex for $1<p<\infty$.
(d) Hanner's inequality holds in $\Lpp$ for $1\leq p<\infty$.
(e) Homogeneity of norm for $1\leq p<\infty$:  
If $f\in\Sc'$ then
$f\in\Lpp$ if and only if $\tau_xf\in\Lpp$ for each $x\in\RR$.
If $f\in\Lpp$ then $\norm{\tau_xf}'_p=\norm{f}'_p$.
(f) Continuity in norm for $1\leq p<\infty$: $\lim_{x\to 0}\norm{\tau_xf-f}'_p=0$ for
each $f\in\Lpp$.
(g) Gateaux derivative: Let $f,g\in\Lpp$ for $1<p<\infty$.  Define
$M(t)=\norm{f+tg}'_p$.  Then $M$ is differentiable and 
$|M'(t)|\leq\norm{g}'_p$.
\end{theorem}

\begin{proof}
(a) See \cite[p.~245]{folland}.
(b) See \cite[p.~126]{yosida}  and \cite{clarkson}.
(c) See \cite[p.~112]{rudinreal}.
(d) See \cite[p.~49]{liebloss}.
(e) See \cite[p.~182]{rudinreal}.
(f) See \cite[p.~182]{rudinreal}.
(g) See \cite[p.~51]{liebloss}.\end{proof}

Each $L^p$ space is a Dedekind complete Banach lattice 
under the partial order: $F\leq G$ if and only if
$F(x)\leq G(x)$ for almost all $x\in\RR$.
This lattice structure is inherited by $\Lpp$.
Here
we list only a few lattice properties shared by
$L^p$ and $\Lpp$.  See \cite{leader} for the lattice defined
by primitives of Henstock--Kurzweil integrable functions.
\begin{theorem}\label{theoremlattice}
In $\Lpp$ define $f\preceq g$ if and only if $F\leq G$, where
$F$ and $G$ are the respective primitives in $L^p$.  (a) $\Lpp$ is a
Banach lattice, Dedekind complete and lattice isomorphic to $L^p$.  (b)
$|f|=|D(F)|=D|F|$ and $\norm{|f|}'_p=\norm{f}'_p$.
(c) The space $\Lany{1}$ is an abstract $L$-space. 
\end{theorem}

\begin{proof}
See \cite[XII.2, XII.3]{yosida}.  \end{proof}

A version of the dominated convergence theorem
is then the following \cite[Theorems~7.2, 7.8]{bartleelements}.
\begin{prop}
Let $1\leq p<\infty$.  Let $(f_n)\subset\Lpp$ with respective primitives
$(F_n)\subset L^p$.  Suppose there is a measurable function $F$ such that
$F_n\to F$ almost everywhere or in measure.  Suppose there is
$g\in\Lpp$ such that $|f_n|\preceq g$ in $\Lpp$.  Then $F'\in\Lpp$ and
$\lim_{n\to\infty}\norm{f_n-F'}'_p=0$.
\end{prop}

If $F\in L^p$ and $g\in L^q$ with $G(x)=\int_0^xg(t)\,dt$ then
$\intinf F(x)g(x)\,dx$ can be computed using an increasing
sequence of step functions as
follows.  First write $F=F^+-F^-$ and $g=g^+-g^-$.  The product
$Fg$ then is a linear combination of four products of positive
functions.  Hence, it suffices to consider the case $F\geq 0$ and
$g\geq 0$.  In the following we take a supremum over step
functions $\sigma\leq F$ where $\sigma=\sum_1^{N_\sigma} a_n^{(\sigma)}
\chi_{(x_n^{(\sigma)},y_n^{(\sigma)})}$.  Here, $\sigma \leq F$ means
$\sigma(x)\leq F(x)$ for almost all $x\in \RR$.  Then
$\intinf F(x)g(x)\,dx=\sup_{\sigma\leq F}\sum_1^{N_\sigma} a_n^{(\sigma)}
[G(y_n^{(\sigma)})-G(x_n^{(\sigma)})]$.  But this is equivalent to
$\intinf fG=-\sup_{\sigma'\preceq f}\intinf \sigma'G$. And,
for $\sigma=\sum_1^Na_n\chi_{(x_n,y_n)}$ we have
\begin{eqnarray*}
\intinf \sigma'G & = & \intinf\sum_1^Na_n[\tau_{x_n}\delta-\tau_{y_n}\delta]G
=\sum_1^Na_n\left[\intinf (\tau_{x_n}\delta)G-\intinf(\tau_{y_n}\delta)G
\right]\\
 & = & \sum_1^Na_n(\langle\tau_{x_n}\delta,G\rangle-
\langle\tau_{y_n}\delta,G\rangle)
=\sum_1^Na_n[G(x_n)-G(y_n)].
\end{eqnarray*}
This then furnishes an equivalent definition of the integral.

\section{Examples}\label{sectionexamples}
In this section we give examples of functions in $\Lpp$ that are not
in any $L^p$ space.  We also show that $\Lpp$ contains some functions that
are not in $L^1_{loc}$ and some functions that have conditionally convergent
integrals.
A simple integral condition is given in
Proposition~\ref{propintcondition} that ensures a function is in $\Lpp$.
It is shown that differences of translated Dirac distributions can
be in $\Lpp$.  Distributions in $\Lpp$ can have primitives that
have no pointwise derivative at any point or a pointwise derivative
that vanishes almost everywhere.

First
note that $\Lpp$ contains many functions (i.e. regular distributions).
For example,
$\Sc\subset L^p$ for each $1\leq p<\infty$.  Let $F\in\Sc$ and
let $G\in I^q$ for any $1<q\leq\infty$.  Then 
$\intinf F'G=\intinf F'(x)G(x)\,dx=-\intinf F(x)G'(x)\,dx$.  Hence,
$\Sc\subset\Lpp$.  If
$f\in L{\!}'^{\,p}\cap L{\!}'^{\,q}$ then $\norm{f}'_p$ and $\norm{f}'_q$
may well be different.

\begin{prop}\label{propintcondition}
Let $f\fn\RR\to\RR$ such that (a) $\intinf f(t)\,dt=0$ 
and (b) $\intinf |t|^\alpha f(t)\,dt$
exists for some $\alpha>1/p$.  Then $f\in\Lpp$.
\end{prop}

\begin{proof}
Let $F(x)=\int_{-\infty}^x f(t)\,dt$.  Then $F$ is
continuous on $\Rbar$.  Let $M>0$.  By the second mean value
theorem for integrals \cite{celidze} there is $\xi\geq M$ such that
$$
\int_{M}^\infty|F(x)|^p\,dx  =  \int_M^\infty\left|\int_x^\infty
f(t)\,dt\right|^p dx
  =  \int_M^\infty|x|^{-\alpha p}\left|\int_{x}^\xi
t^\alpha f(t)\,dt\right|^p dx,
$$
which is finite.
Similarly, $\int_{-\infty}^{-M}|F(x)|^p\,dx<\infty$.\end{proof}

Condition (a) is here interpreted as a Henstock--Kurzweil integral but
can be interpreted as a Lebesgue or wide Denjoy integral.  See
\cite{celidze}.  A sufficient condition for (b) is that
$f(t)=O(|t|^{-\beta})$ as $|t|\to\infty$ for some $\beta>1+1/p$.
For example, let $f(x)=\sin(x)/|x|$.  Then condition (a) is satisfied
as a Henstock--Kurzweil or improper Riemann integral.
From (b) we see that $f\in\Lpp$ for all $p>1$.  And, let $g(x)=
x(|x|+1)^{-\gamma}=O(|x|^{1-\gamma})$ as $|x|\to\infty$.  Condition
(a) is satisfied if $\gamma>2$.  Then $g\in\Lpp$ for all
$p>(\gamma-2)^{-1}$.

The Dirac distribution is not in $\Lpp$ since $\delta=H'$ and for no
$1\leq p<\infty$ is $H\in L^p$.  However, the Heaviside step function
is regulated, i.e., it has a left and right limit at each point. The
Dirac distribution then has a regulated primitive integral.  See
\cite{talvilaregulated}.
Differences
of translated Dirac distributions may be in $\Lpp$.  Let $F=\chi_{(a,b)}$.
Then $F\in L^p$ for each $1\leq p<\infty$.  Hence, $f=F'=\tau_a\delta
-\tau_b\delta\in \Lpp$.  Take any $1<q\leq\infty$ and let $g\in L^q$.
Define $G(x)=\int_0^xg(t)\,dt$.
Then $\intinf fG=-\int_a^bG'(x)\,dx=
-\int_a^bg(t)\,dt$.

Let $F(x)=|x|^{-\gamma}e^{-|x|}$. If $0<\gamma<1/p$ then $F\in L^p$. 
Let $f(x)=F'(x)=-\textrm{sgn}(x)(1+\gamma/|x|)|x|^{-\gamma}
e^{-|x|}$ for $x\not=0$.  Since 
$f(x)\sim -\gamma\,\textrm{sgn}(x)\,|x|^{-(\gamma+1)}$ as $x\to 0$ it follows
that $f\in\Lpp$ but $f\not\in L^1_{loc}$.  Note that $f$ is integrable
in the principal value sense.
Similarly, if $G(x)=\log\abs{x}e^{-|x|}$ then
$G'\in\Lpp$ for each $1\leq p<\infty$ and $G'(x)\sim
1/x$ as $x\to 0$.  This last is also integrable in the principal
value sense.

Let $F(x)=\sin(\exp(\abs{x}^3))/(x^2+1)$.  Then $F\in C^\infty(\RR)\cap L^p$ 
for each $1\leq p<\infty$.
Let $f(x)=F'(x)$.
Then $f\in\Lpp$ but for no
$1\leq p\leq\infty$ is $f\in L^p$.  Note that $f$ is Henstock--Kurzweil
integrable,
improper Riemann integrable, and that $f(x)=O(\exp(\abs{x}^3))$
as $\abs{x}\to\infty$.

Let $F(x)=x^2\sin(x^{-4})$ with $F(0)=0$.  Then $F\in C(\RR)\cap L^p$
for each $1\leq p<\infty$.  The derivative exists at each point
and
\begin{eqnarray*}
F'(x)=f(x) & = & \left\{\!\!\begin{array}{cl}
2x\sin(x^{-4})-4x^{-3}\cos(x^{-4}), & x\not=0\\
0, & x=0\\
\end{array}
\right.\\
 & \sim & \left\{\!\!\begin{array}{cl}
-4x^{-3}\cos(x^{-4}), & x\to 0\\
x^{-3}(2-4\cos(x^{-4}), & |x|\to\infty.
\end{array}
\right.
\end{eqnarray*}
It follows that $f$ is in each space $\Lpp$ but $\int_0^1|f(x)|^p\,dx$ diverges for each $1\leq p<\infty$
so $f$ is not in any $L^p$ space.
However, $f$ is integrable in the
Henstock--Kurzweil and improper Riemann sense.

Let $F(x)=x^{-\gamma}H(x)H(1-x)$ and suppose $0<\gamma<1/p$.  Then $F\in L^p$.
And, $F=F_1-F_2$ where
$F_1(x)=x^{-\gamma}H(x)$ and $F_2(x)=x^{-\gamma}H(x-1)$.
Define $f=F'\in \Lpp$.  To find an explicit formula for $f$ let
$\phi\in\Sc$.  Then
$$
\langle F_1',\phi\rangle  =  
-\lim_{\epsilon\to 0^+}\int_\epsilon^\infty x^{-\gamma}\phi'(x)\,dx
  =  \lim_{\epsilon\to 0+}\left[\epsilon^{-\gamma}\phi(0)-\gamma
\int_\epsilon^\infty x^{-(\gamma+1)}\phi(x)\,dx\right].
$$
And, $F_1'$ is the Hadamard finite part of the
divergent integral $\int_0^\infty x^{-(\gamma+1)}\phi(x)\,dx$.
See \cite[\S2.5]{zemanian} for details.  To find $F_2'$ let
$p_\gamma(x)=x^\gamma$.  Then
$$
\langle F_2',\phi\rangle  =  -\int_1^\infty x^{-\gamma}\phi'(x)\,dx
  =  \phi(1)-\gamma\int_1^\infty x^{-(\gamma+1)}\phi(x)\,dx.
$$
This shows that $F_2'=\tau_1\delta-\gamma(\tau_1H)p_{-(\gamma+1)}$.
Notice that the pointwise derivative of $F$ is
$F'(x)=-\gamma\, x^{-(\gamma+1)}$ for $0<x<1$, $F'(x)=0$ for
$x<0$ and $x>1$, and $F'(x)$ does not exist for $x=0,1$.
Hence, this pointwise derivative has a non-integrable singularity at
$x=0$, i.e., it is not in $L^1_{loc}$.

Let $E\subset\RR$ be a set of finite measure.  Then $\chi_E\in L^q$
for each $1\leq q\leq\infty$.  Define $g\in L^q$ by $g=\chi_E$
and define $G\in I^q$ by $G(x)=\int_0^xg(t)\,dt$.  Let $1\leq p<\infty$
with conjugate exponent $q$.  Let $f\in\Lpp$ with primitive $F\in L^p$.
Then $\intinf fG=-\int_E F$.  In particular, if $E$ is an interval with
endpoints $a<b$ then $G(x)=0$ for $x\leq a$, $G(x)=x-a$ for $a\leq x\leq b$
and $G(x)=b-a$
for $x\geq b$.  And, $G-G(a)\in I^q$ so that $\intinf fG=-\int_a^bF$.
In general, $\int_a^b f$ does not exist.

Let $\balexc$ be the functions in $C(\Rbar)$ that vanish at
$-\infty$.  Define $\alexc=\{f\in\Sc'\mid f=F' \text{ for some }
F\in\balexc\}$.  If $f\in\alexc$ then it has a unique
primitive $F\in\balexc$ such that $F'=f$.  The continuous
primitive integral of $f$ is $\int_a^b f=F(b)-F(a)$ for all $-\infty\leq a<
b\leq\infty$.  See \cite{talviladenjoy} for details.
The next two examples are distributions that have a continuous
primitive integral.

Let $\sigma\fn\RR\to[0,1]$ be a continuous function such
that $\sigma'(x)=0$ for almost all $x\in\RR$.  Define
$F(x)=\exp(-x^2)\sigma(x)$ then $F\in L^p$ for each $1\leq p<\infty$.
The pointwise derivative is $F'(x)=0$ for almost all $x\in\RR$.  For
each $1\leq q\leq \infty$, if
$\psi\in L^q$ we have the Lebesgue integral $\intinf F'(x)\psi(x)\,dx=0$.
Now define $f=F'\in\Lpp$.  Then for each $G\in I^q$ with $q$ conjugate
to $p$ we have the $L^p$ primitive integral 
$\intinf fG=-\intinf F(x)G'(x)\,dx$.  This need not be $0$.

Let $\omega\fn\RR\to\RR$ be a bounded continuous function such
that the pointwise derivative $\omega'(x)$ exists for no $x\in\RR$.  Define
$F(x)=\exp(-x^2)\omega(x)$ then $F\in L^p$ for each $1\leq p<\infty$.
The pointwise derivative $F'(x)$ exists nowhere so for no function
$\psi$ does
the Lebesgue integral $\intinf F'(x)\psi(x)\,dx$ exist.
Now define $f=F'\in\Lpp$.  Then for each $G\in I^q$ with $q$ conjugate
to $p$ we have the $L^p$ primitive integral 
$\intinf fG=-\intinf F(x)G'(x)\,dx$.  This last exists as a Lebesgue
integral.

\section{Convolution}
The convolution of functions $f$ and $g$ is
$f\ast g(x)=\intinf f(x-y)g(y)\,dy$.  In the literature, various conditions
have been imposed on $f$ and $g$ for existence of the convolution.  
For instance, 
for each $1\leq p\leq \infty$ with $q$ conjugate,
the convolution is a bounded linear operator 
$\ast\fn L^p\times L^q\to L^\infty$ such that
$f\ast g$ is uniformly continuous and
$\norm{f\ast g}_\infty\leq\norm{f}_p\norm{g}_q$.
If $1\leq p,q,r\leq\infty$ such that $1/p+1/q=1+1/r$.
Then it is a bounded linear operator
$\ast\fn L^p\times L^q\to L^r$ such that
$\norm{f \ast g}_r\leq \norm{f}_p\norm{g}_q$.
See \cite{folland} for the $L^p$ theory of convolutions.

Many of the $L^p$
results have analogues in $\Lpp$.  Use the notation ${\tilde \phi}(x)
=\phi(-x)$ for function $\phi$.  First we consider the convolution as
$\ast\fn \Lpp\times I^q\to L^\infty$.  This
can be defined directly using the integral 
$f\ast G(x)=\intinf {\widetilde{\tau_xf}}\,G$, where 
$\langle {\widetilde{\tau_xf}},\phi\rangle =\langle f, {\widetilde{\tau_x\phi}}
\rangle$ for distribution
$f$ and test function $\phi$.
We will write this as $f\ast G(x)=\intinf f(x-y)G(y)\,dy$ even though it is an
integral in $\Lpp$.  To distinguish from Lebesgue integrals, in 
this section if $f\in\Lpp$ then we will always write its primitive as
$F\in L^p$.
Further in this section we
define the convolution $\ast\fn\Lany{p}\times L^q\to\Lany{r}$ 
and $\ast\fn L^p\times \Lany{q}\to\Lany{r}$ 
using a limiting
procedure and the density of the compactly supported smooth functions
in $L^q$.

\begin{theorem}\label{theoremconvinfty}
For $1\leq p<\infty$,
let $f\in\Lpp$ with primitive $F\in L^p$.  Let $G\in I^q$ such
that $G(x)=\int_0^xg(t)\,dt$ for some $g\in L^q$, where $q$ is conjugate
to $p$.  Define the
convolution $f\ast G(x)=\intinf f(x-y)G(y)\,dy$.
The convolution has
the following properties. (a) 
$\ast\fn\Lpp\times I^q\to L^\infty$ such that
$f\ast G=F\ast g=g\ast F$,
$\norm{f\ast G}_\infty\leq \norm{f}'_p\norm{G}_{I,q}$
and $f\ast G$ is uniformly continuous on $\RR$.
(b) $f\ast k=0$ for each $k\in\RR$.
(c) Suppose that $f\in\Lonep$ and $g\in L^\infty$.
If $h\in L^1$ such that $\intinf |y\,h(y)|\,dy<\infty$
then $f\ast (G\ast h)=(f\ast G)\ast h$.
(d) For each $z\in \RR$, $\tau_z(f\ast G)=(\tau_zf)\ast G=f\ast(\tau_zG)$.
(e) For each $f\in\Lpp$ define $\Phi_f\fn I^q\to L^\infty$ by
$\Phi_f[G]=f\ast G$.  Then $\Phi_f$ is a bounded linear operator
and $\norm{\Phi_f}= \norm{f}'_p$. 
For each $G\in I^q$ define 
$\Psi_G\fn \Lpp\to L^\infty$ by
$\Psi_G[f]=f\ast G$.  Then $\Psi_G$ is a bounded linear operator.
For $1<p<\infty$ we have $\norm{\Psi_G}=\norm{G}_{I,q}$. 
For $p=1$ we have $\norm{\Psi_G}\leq\norm{G}_{I,\infty}$. 
(f) $\textrm{supp}(f\ast G)\subset \textrm{cl}(\textrm{supp}(F)+\textrm{supp}(g))$.
\end{theorem}
\begin{proof}
Most parts of the theorem are proved by reverting
to equivalent $L^p$ results.  These are proved in \cite{folland}.
(a)  Note that $\tau_x(DF)=D(\tau_xF)$ so 
$f\ast G(x)=\intinf f(x-y)G(y)\,dy=\intinf F(x-y)g(y)\,dy$
by Definition~\ref{defnintegral}.
The H\"older inequality now gives $\norm{f\ast G}_\infty\leq
\norm{F}_1\norm{g}_\infty$.  Uniform continuity follows \cite[p.~241]{folland}.
(b) See Definition~\ref{defnintegral}.
(c) Note that $G\ast h$ exists on $\RR$ by the H\"older inequality since
$G(y)/(|y|+1)$ is bounded and $(|y|+1)h(y)$ is in $L^1$.  By dominated
convergence $(G\ast h)'(x)=g\ast h(x)$ for each $x\in\RR$.  Hence,
$\norm{(G\ast h)'}_\infty\leq \norm{g}_\infty\norm{h}_1$.  Thus,
$G\ast h\in I^\infty$ (modulo a constant).  
We then have $f\ast(G\ast h)=F\ast(G\ast h)'
=F\ast(g\ast h)=(F\ast g)\ast h=(f\ast G)\ast h$.  Associativity
for Lebesgue integrals is proved in \cite[p.~240]{folland}.
(d) $\tau_z(f\ast G)=\tau_z(F\ast g)=(\tau_zF)\ast g=(\tau_zf\ast G)
=F\ast(\tau_zg)=f\ast(\tau_zG)$.  Translation for functions is proved
in \cite[p~240]{folland}.
(e) Note that 
$$
\norm{\Phi_f}=\sup_{\norm{G}_{I,q}=1}\norm{f\ast G}_\infty
=\sup_{\norm{g}_q=1}\norm{F\ast g}_\infty
\leq\sup_{\norm{g}_q=1}\norm{F}_p\norm{g}_q
= \norm{F}_p.\label{Phi}
$$
We get equality in the H\"older inequality \cite[p.~46]{liebloss} by taking
$g(y)=\sgn(F(-y))$ when $p=1$ and otherwise
$g(y)=
\norm{F}_p^{1-p}\sgn{(F(-y))}|F(-y)|^{p/q}$.
Then $F\ast g(0)=\norm{F}_p$.
The proof for $\Psi$ is similar.
(f) See \cite[p.~240]{folland}.
\end{proof}

Note that if $F\in L^p$ and $G\in I^q$ then $F\ast G$ need not
exist as a function at any point.  For example, let 
$F(x)=|x|^{-4/(3p)}$
for $|x|>1$ and $F(x)=0$, otherwise.  Take $g(x)=|x|^{-(1/q+1/(3p))}$ 
for $|x|>1$ and $g(x)=0$, otherwise.  For each $1\leq p<\infty$ we
have $F\in L^p$ and $g\in L^q$.  Let $G(x)=\int_0^xg$.  Then the
Lebesgue integral defining $F\ast G(x)$ diverges for each $x\in\RR$.
Hence, we cannot define $f\ast G$ as $(F\ast G)'$. 

From \cite[p.~46]{liebloss}, we get the equality $\norm{F\ast g}_\infty=
\norm{F}_1\norm{g}_\infty$ if and only if $|g(y)|=\textrm{sgn}(F(x-y))$ for
some $x$ and almost all $y$.  Hence, in (e) the operator norm of $\Psi_G$ need
not equal $\norm{G}_{I,\infty}$.

For tempered distributions, the convolution can be defined as
$\ast\fn\Sc'\times\Sc\to C^\infty$ by 
$T\ast\phi(x)=\langle T,\tau_x\tilde{\phi}\rangle$.
See \cite{folland}, where another equivalent definition is also given.
Let $\phi\in\Sc$.
If $f\in\Lpp$ with primitive $F\in L^p$ then
according to this definition, 
\begin{eqnarray*}
f\ast \phi(x) & = & -\langle F,(\tau_x\tilde{\phi})'\rangle
=-\intinf F(y)\frac{\partial \phi(x-y)}{\partial y}\,dy\\
 & = & \intinf F(y)\phi'(x-y)\,dy=F\ast\phi'(x).
\end{eqnarray*}
Since $\Sc\subset L^q$ for each $1\leq q\leq\infty$ and $\Sc$ is closed
under differentiation, if $\Phi\in\Sc$ then the function
$x\mapsto\int_0^x\Phi'(t)\,dt$ is in $I^q$ for each $1\leq q\leq\infty$.
The definition in Theorem~\ref{theoremconvinfty} then gives
$f\ast \Phi(x)=F\ast\Phi'(x)$, since convolution with a constant is zero.
Our definition then agrees with the usual one given above when we convolve
with the integral of a test function in $\Sc$.

If $f\in\Lpp$ and $g\in L^q$ then the integral
$\intinf f(y)g(x-y)\,dy$ need not exist.  However, 
$C^\infty_c$, the set of smooth functions with compact support,
is dense in $L^q$ and we can use the
definition in Theorem~\ref{theoremconvinfty} to 
define $f\ast g$ with a
sequence in $C^\infty_c$
that converges to a function in $L^q$.
\begin{definition}\label{defnconvp}
Let $p,q,r\in[1,\infty)$ such that $1/p+1/q=1+1/r$.
Let $f\in\Lpp$ with primitive $F\in L^p$ 
and let $g\in L^q$.  Let $(g_n)\subset C^\infty_c$
such that $\norm{g_n-g}_q\to 0$.  Define $f\ast g$ to be the
unique element in $\Lany{r}$ such that $\norm{f\ast g-f\ast g_n}'_r\to 0$.
Define $F\ast g'$ to be the
unique element in $\Lany{r}$ such that $\norm{F\ast g'-F\ast g_n'}'_r\to 0$.
\end{definition}

To see that the definition makes sense, suppose $\supp(g_n)\subset
[a_n,b_n]$.  Let $x\in(\alpha,\beta)$.  Then
$(F\ast g_n)'(x)=\frac{d}{dx}\int_{\alpha-b_n}^{\beta-a_n}F(y)g_n(x-y)\,dy=
\intinf F(y)g_n'(x-y)\,dy$ by dominated convergence
since $F\in L^1_{\textrm{loc}}$ and $g_n'\in L^\infty$.  
Then, from Theorem~\ref{theoremconvinfty}(a),
$$
\norm{f\ast g_n}'_r =\norm{F\ast g_n'}'_r=\norm{F\ast g_n}_r\leq\norm{F}_p
\norm{g_n}_q.
$$
From this it follows that
$
\norm{f\ast g_n-f\ast g_m}'_r = \norm{f\ast(g_n-g_m)}'_r
\leq \norm{f}'_p\norm{g_n-g_m}_q\to 0$ as  $m,n\to\infty$.
Hence, $(f\ast g_n)$ is a Cauchy sequence in the complete space
$\Lany{r}$.  It therefore converges to an element of $\Lany{r}$ which
we label $f\ast g$.  Similarly with $(F\ast g_n')$.  
This also shows that $f\ast g$ and $F\ast g'$ are independent
of the choice of sequence $(g_n)$.

The following theorem shows that convolutions in $\Lpp\times L^q$
and $L^p\times\Lany{q}$
share many of the properties of convolutions in $L^p\times L^q$.
Part (b) extends Young's inequality \cite[p.~241]{folland}.
Part (g) shows that convolutions with $L^p$ functions can be used to approximate
distributions in $\Lpp$.

\begin{theorem}\label{theoremconvLp}
Let $p,q,r\in[1,\infty)$ such that $1/p+1/q=1+1/r$.
Let $f\in\Lpp$ with primitive $F\in L^p$ 
and let $g\in L^q$.
Then  $\ast\fn \Lpp\times L^q\to \Lany{r}$ 
and $\ast\fn L^p\times \Lany{q}\to\Lany{r}$
with the following properties.
(a) $f\ast g=(F\ast g)'=F\ast g'$.
(b) $\norm{f\ast g}'_r=\norm{F\ast g'}'_r\leq \norm{F}_p\norm{g}_q$.
(c)  If $h\in L^1$
then $f\ast (g\ast h)=(f\ast g)\ast h$ and 
$F\ast(g'\ast h)=(F\ast g')\ast h$.
(d) For each $z\in \RR$, $\tau_z(f\ast g)=(\tau_zf)\ast g=f\ast(\tau_zg)
=(\tau_zF)\ast g'=F\ast\tau_zg'$.
(e) For each $f\in\Lpp$ define $\Phi_f\fn L^q\to \Lany{r}$ by
$\Phi_f[g]=f\ast g$.  Then $\Phi_f$ is a bounded linear operator
and $\norm{\Phi_f}\leq \norm{f}'_p$. 
For each $g\in L^q$ define 
$\Psi_g\fn \Lpp\to \Lany{r}$ by
$\Psi_g[f]=f\ast g$.  Then $\Psi_g$ is a bounded linear operator
and $\norm{\Psi_g}\leq\norm{g}_q$. 
For each $F\in L^p$ define $A_F\fn \Lany{q}\to \Lany{r}$ by
$A_F[g']=F\ast g'$ for each $g\in L^q$.  Then $A_F$ is a bounded linear operator
and $\norm{A_F}\leq \norm{F}_p$. 
For each $g\in L^q$ define 
$B_{g'}\fn L^p\to \Lany{r}$ by
$B_{g'}[F]=F\ast g'$.  Then $B_{g'}$ is a bounded linear operator
and $\norm{B_{g'}}\leq\norm{g'}'_q$. 
(f) $\textrm{supp}(f\ast g)\subset \textrm{cl}(\textrm{supp}(f)+\textrm{supp}(g))$.
(g) Let $1\leq p<\infty$.  Let $g\in L^1$.  Define $F_t(x)=F(x/t)/t$ for $t>0$.
Let $a=\intinf F_t(x)\,dx=\intinf F(x)\,dx$.  Then $\norm{F_t\ast g'-ag'}'_p\to 0$
as $t\to 0$.
\end{theorem}

\begin{proof}
(a) Take $(g_n)\subset C^\infty_c$
such that $\norm{g_n-g}_q\to 0$.  Then by Theorem~\ref{theoremconvinfty}(a)
and the paragraph following Definition~\ref{defnconvp},
$\norm{f\ast g_n-(F\ast g)'}'_r=\norm{F\ast g_n'-(F\ast g)'}'_r=
\norm{F\ast(g_n-g)}_r\leq
\norm{F}_p\norm{g_n-g}_q\to 0$ as $n\to\infty$.
(b) $\norm{f\ast g}'_r=\norm{(F\ast g)'}'_r=\norm{F\ast g'}'_r=\norm{F\ast g}_r
\leq\norm{F}_p\norm{g}_q$.
(c) $(f\ast g)\ast h=(F\ast g)'\ast h=[(F\ast g)\ast h]'=
[F\ast (g\ast h)]'=F'\ast(g\ast h)=f\ast(g\ast h)$.
The $L^p$ result used is in \cite[p.~240]{folland}.
The other case is similar.
(d) $\tau_z(f\ast g)=\tau_zD(F\ast g)=D[\tau_z(F\ast g)]
=D[(\tau_zF)\ast g]=(\tau_zF)\ast g'=[D(\tau_zF)]\ast g=(\tau_zf)\ast g
=D[F\ast\tau_zg]=f\ast \tau_zg=F\ast D(\tau_zg)=F\ast\tau_zg'$.  
The $L^p$ results used
are in \cite[p.~240]{folland}.
(e)  Note that 
$$
\norm{\Phi_f}=\sup_{\norm{g}_q=1}\norm{f\ast g}'_r
=\sup_{\norm{g}_q=1}\norm{F\ast g}_r\leq
\sup_{\norm{g}_q=1}\norm{F}_p\norm{g}_q
=\norm{F}_p.
$$
The other cases are similar.
(f) Zemanian \cite[5.4(2)]{zemanian} uses the
definition
$$
\langle F\ast g,\phi\rangle = \intinf F(y)\intinf g(x)\phi(x+y)\,dx\,dy
=\intinf F{\widetilde{g\ast \phi}}
$$
for $\phi\in\Sc$.  The integral exists for $F\in L^p$ and $g\in L^q$
since $\phi\in L^1$ so $g\ast \phi\in L^q$ (Young's inequality).
Whereas, Definition~\ref{defnconvp} and (a) of this theorem give
\begin{eqnarray*}
\langle f\ast g,\phi\rangle & = & -\langle F\ast g,\phi'\rangle
=-\intinf \intinf F(y)g(x-y)\,dy\, \phi'(x)\,dx\\
 & = & -\intinf F(y)\intinf g(x-y)\phi'(x)\,dx\,dy\\
 & = & -\intinf F(y)\intinf g(x)\phi'(x+y)\,dx\,dy.
\end{eqnarray*}
The Fubini--Tonelli theorem justifies reversing the order of integration.
This proves the equivalence of our Definition~\ref{defnconvp} with
Zemanian's definition.  The support result then follows from
Theorem~5.4-2 and Theorem~5.3-1 in \cite{zemanian} and the fact that
the support of a function is the closure of the set on which it does
not vanish.
(g) $\norm{F_t\ast g'-ag'}'_p=\norm{D(F_t\ast g-ag)}'_p
=\norm{F_t\ast g-ag}_p\to 0$ as $t\to 0$.  See \cite[p.~242]{folland}.\end{proof}

If $F,G\in L^1$ then $F\ast G\in L^1$.  Hence, we can define a Banach algebra
by defining $F'\star G'=(F\ast G)'$.
\begin{theorem}\label{banachalgebra}
Let $f,g\in \Lany{1}$ with respective primitives $F,G\in L^1$.  Define 
the product $\star\fn \Lany{1}\times\Lany{1}\to\Lany{1}$ by
$f\star g=(F\ast G)'$.  Then $\Lany{1}$ is a Banach algebra isometrically
isomorphic to
the convolution algebra on $L^1$.
\end{theorem}

\begin{proof}
The proof is elementary.  For example, to show commutativity,
$f\star g=F'\star G'=(F\ast G)'=(G\ast F)'=G'\star F'=g\star f$.
\end{proof}

This product is not compatible with the convolution defined in
Theorem~\ref{theoremconvLp}.  For example, let $g(x)=-2x\exp(-x^2)$.  
Then $g\in L^1$.  Define $G(x)=\exp(-x^2)$.  Then $g(x)=G'(x)$ for
all $x$ and $G\in L^1$ so $g\in \Lany{1}$.  Let $F=\chi_{(0,1)}\in L^1$
and $f=F'=\delta_0-\delta_1\in\Lany{1}$.  A calculation shows
$f\ast g(x)=2(x-1)\exp(-[x-1]^2)-2x\exp(-x^2)$ while
$f\star g(x)=\exp(-x^2)-\exp(-[x-1]^2)$.

\section{Fourier transform in $\Lany{1}$}
If $F\in L^1$ then its Fourier transform is $\Fhat(s)=\intinf e^{-isx}F(x)\,dx$.
This Lebesgue integral converges for each $s\in\RR$.
It is known that $\Fhat$ is continuous on $\Rbar$ and vanishes at $\pm\infty$
(Riemann--Lebesgue lemma).
This defines a linear operator $\ \hat{}: L^1\to C(\Rbar)$.
Also, each tempered distribution has a Fourier transform that is also
a tempered distribution.  If $T\in\Sc'$ then $\langle \That,\phi\rangle
=\langle T,\phihat\rangle$.  The most important properties of Fourier
transforms in $L^1$ and $\Sc'$ can be found in \cite{folland}.

Since the complex exponential is in $I^\infty$ we can define the
Fourier transform
in $\Lonep$ using the integral definition.
It then shares many of the properties of $L^1$ transforms.

\begin{definition}\label{defnFourier}
Let $e_s(x)=e^{-isx}$.  Let $f\in\Lonep$ with primitive $F\in L^1$.
For each $s\in\RR$
define $\fhat(s)=\intinf f e_s$.
\end{definition}
Note that $e_s-1\in I^\infty$ since $e_s(x)=1-is\int_0^x e^{-ist}\,dt$.
The integral $\intinf f e_s$ is then well defined as in 
Definition~\ref{defnintegral}.

If $f=F'\in \Lpp$ and $G\in I^q$ for conjugate $q$ are periodic then the formula
$\int_{-\pi}^\pi F'G=-\int_{-\pi}^\pi F(x)G(x)\,dx$ can be used to
define Fourier coefficients. We will leave a study of $\Lpp$ Fourier
series for elsewhere.

\begin{theorem}\label{theoremFourier}
Let $f\in\Lonep$ with primitive $F\in L^1$.  Let $s\in\RR$.  Then
(a) $\fhat(s)=is\Fhat(s)$.
(b) $\fhat$ is continuous on $\RR$.
(c) $|\fhat(s)|= |s|\norm{\Fhat}_\infty\leq|s|\norm{F}_1=|s|\norm{f}_1'$
(d) $\fhat(s)=o(s)$ as $|s|\to\infty$.
(e) Definition~\ref{defnFourier} agrees with the tempered distribution
definition.
(f) ${\widehat {\tau_yf}}(s)=e^{-isy}\fhat(s)$ for each $y\in\RR$.
(g) Let $K(x)=xF(x)$.  If $F,K\in L^1$ then $\fhat$ is differentiable
and $D\fhat(s)=i\Fhat(s)+s{\hat K}(s)$.
(h) Let $g\in L^1$.  Then ${\widehat{f\ast g}}(s)=\fhat(s)\ghat(s)$.
(i) Let $g\in L^1$ such that the function $s\mapsto sg(s)$ is also
in $L^1$.  Then $\intinf \fhat(s) g(s)\,ds=\intinf f\ghat$.
\end{theorem}

\begin{proof}
Most parts are proved by reverting to an $L^1$ result.
For these, see \cite[\S8.3]{folland}. (e) If $\phi\in\Sc$ then
$\phihat\in\Sc$.  Using the tempered distribution definition,
\begin{align}
&\langle \fhat,\phi\rangle = \langle F',\phihat\rangle
=-\langle F,D\phihat\rangle
=i\intinf F(x)\intinf se^{-ixs}\phi(s)\,ds\,dx\label{FourierLDC}\\
&=i\intinf s\phi(s)\intinf e^{-ixs}F(s)\,dx\,ds
=i\intinf s\phi(s)\Fhat(s)\,ds,\label{FourierFubini}
\end{align}
in agreement with Definition~\ref{defnFourier}.  
Dominated convergence is used in \eqref{FourierLDC} to differentiate
under the integral.  The Fubini--Tonelli theorem allows interchange
of iterated integrals in \eqref{FourierFubini}.
(f) Note that $(\tau_yF)'=\tau_y(F')$.  Then
$\widehat{\tau_yf}(s)=is\,\widehat{\tau_yF}(s)=ise^{-isy}\Fhat(s)=
e^{-isy}\fhat(s)$.
(h) Use Theorem~\ref{theoremconvLp}(a) with $p=q=r=1$ to get
$\widehat{f\ast g}(s)=
\widehat{(\!F\!\!\ast\!g\!)'\,}\!(s)
=is\,\widehat{F\ast g}(s)
=is\Fhat(s)\ghat(s)=\fhat(s)\ghat(s)$.
(i) Note that 
\begin{equation}
\intinf\fhat(s) g(s)\,ds=i\intinf \intinf se^{-isx}F(x)g(s)
\,dx\,ds\label{exchange1}
\end{equation}
and 
\begin{equation}
\intinf f\ghat=-\intinf F(x)\ghat\,'(x)\,dx=
i\intinf\intinf F(x)e^{-isx}sg(s)\,ds\,dx.\label{exchange2}
\end{equation}
The Fubini--Tonelli theorem
gives the equality of the integrals in \eqref{exchange1} and
\eqref{exchange2}.\end{proof}

Thus, $\ \hat{}:\Lonep\to C(\RR)$.  The Riemann--Lebesgue lemma
order relation $\Fhat(s)=o(s)$
as $|s|\to\infty$ is known to be sharp for $F\in L^1$.  Hence, $\fhat$
is a continuous function but
need not be bounded.

\section{Inner product and Fourier transform in $\Lany{2}$}\label{sectionL2}
An important property of the $L^p$ spaces is that $L^2$ is a Hilbert space
with inner product $(f,g)=\intinf f(x)g(x)\,dx$.
If $f,g\in\Lany{2}$ then $\intinf fg$ need not exist.  However,
since the derivative $D\fn L^2\to \Lany{2}$ is a linear isometry,
$\Lany{2}$ is also a Hilbert space.
\begin{theorem}\label{theoremHilbert}
$\Lany{2}$ is a Hilbert space with inner product
$$
(f,g)=\frac{1}{4}\left({\norm{f+g}'_2}^2-{\norm{f-g}'_2}^2\right)
=(F,G),
$$
where $f,g\in\Lany{2}$ with primitives $F,G\in L^2$.
\end{theorem}

\begin{proof}
A necessary and sufficient condition for a Banach space
to be a Hilbert space is that the norm satisfy the parallelogram
identity $\norm{x+y}^2+\norm{x-y}^2=2(\norm{x}^2+\norm{y}^2)$ for
all $x$ and $y$ in the space.  See \cite[I.5]{yosida}.  
The $L^2$ norm then satisfies this
identity and so does $\norm{\cdot}'_2$.  If $f,g\in\Lany{2}$ with
respective primitives $F,G\in L^2$
then 
$$
(f,g)  =  \frac{1}{4}\left({\norm{f+g}'_2}^2-{\norm{f-g}'_2}^2\right)
  =  \frac{1}{4}\left(\norm{F+G}_2^2-\norm{F-G}_2^2\right)
  =  (F,G).
$$
This completes the proof.\end{proof}

Using Theorem~\ref{theoremHilbert} we can define $\intinf fg=
(f,g)=\intinf F(x)G(x)\,dx$ for $f,g\in\Lany{2}$ with primitives
$F,G\in L^2$.

Fourier transforms are defined in $L^2$ using the Plancherel theorem.
If $F\in L^2$ let $(F_n)\subset L^2\cap L^1$ such that $\norm{F_n-F}_2
\to 0$.  Each function $\Fnhat$ is well defined by the usual Fourier
integral formula as an element of $L^2$.
There is a unique function $\Fhat\in L^2$ such that
$\norm{\Fnhat - \Fhat}_2\to 0$.  It follows that the Fourier transform
$\ \hat{}: L^2\to L^2$ is a unitary isomorphism, $\norm{F}_2=\norm{\Fhat}_2$
and if $G\in L^2$ then the Parseval equality is 
$(F,G)=\intinf F(x)G(x)\,dx=[1/(2\pi)]\intinf \Fhat(s)\Ghat(s)\,ds=
(\Fhat, \Ghat)$.  The norm is as usual defined by $\norm{\Fhat}_2=
(\Fhat,\Fhat)^{1/2}$.
See \cite{folland}
for details.  If $F\in L^2$ then the defining sequence in $L^2\cap L^1$
can be taken
as $F_n=\chi_{(-n,n)}F$.

The  derivative is a unitary mapping between $L^2$ and $\Ltwop$.
Suppose $F\in L^2$, $(F_n)\subset L^2$ such that $\norm{F_n-F}_2\to 0$.
Then $\Fhat\in L^2$ and $D\Fhat\in\Ltwop$ is defined as a tempered
distribution by
$
\langle D\Fhat,\phi\rangle=-\langle\Fhat,\phi'\rangle=-\langle F,
\widehat{D\phi}\rangle$
for all $\phi\in\Sc$.
And, if $G\in I^2$ then $\intinf (D\Fhat)G=-\intinf \Fhat(s)G'(s)\,ds$.
Note that $\norm{F_n-F}_2=\norm{\Fnhat-\Fhat}_2=\norm{D\Fnhat-D\Fhat}'_2\to 0$.
If $g\in L^2$ the Parseval equality takes the form
$(F,g)=\intinf F(x)g(x)\,dx=(\Fhat,\ghat)=
[1/(2\pi)]\intinf \Fhat(s)\ghat(s)\,ds=
(D\Fhat, D\ghat)=[1/(2\pi)]\intinf(D\Fhat)(D\ghat)
=(F',g')=\intinf F'g'$.

If $F\in L^2$ then $DF\in \Ltwop\subset\Sc'$ so $\widehat{DF}$ is
defined as a tempered distribution by
$
\langle \widehat{DF},\phi\rangle=\langle DF,\phihat\rangle=-\langle
F,D\phihat\rangle.
$

Since the Fourier transform does not commute with the derivative,
$D\Fhat$ and $\widehat{DF}$ are not necessarily equal.
For example, let $F=\chi_{(-1,1)}\in L^2\cap L^1$.  Then
$\Fhat\in L^2$ and $D\Fhat\in\Ltwop$.  We have
$$
\Fhat(s)=2\int_0^1\cos(sx)\,ds=\left\{\!\!\begin{array}{cl}
\frac{2}{s}\sin(s), & s\not=0\\
2, & s=0.
\end{array}
\right.
$$
The pointwise derivative is then
$$
D\Fhat(s)=\left\{\!\!\begin{array}{cl}
-\frac{2}{s^2}\sin(s)+\frac{2}{s}\cos(s), & s\not=0\\
0, & s=0.
\end{array}
\right.
$$
Whereas, $DF=\tau_{-1}\delta-\tau_1\delta\in \Lonep$ so
$
\widehat{DF}(s)=is\Fhat(s)=2i\sin(s).$  Hence, $D\Fhat\not=\widehat{DF}$.
Note that $\widehat{DF}\in L^\infty$ but is not in any $L^p$ or $\Lpp$ space for
$1\leq p<\infty$.

\section{Higher derivatives}\label{sectionhigher}
In \cite{talvilaacrn}, an integral that inverts higher derivatives of
continuous or regulated primitives was introduced.  Define
$\acn=\{f\in\Sc'\mid f=D^nF \text{ for some } F\in\balexc\}$ where
$\balexc=\{F\in C(\Rbar)\mid F(-\infty)=0\}$ and
$\arn=\{f\in\Sc'\mid f=D^nF \text{ for some } F\in\balexr\}$
where $\balexr$ are the regulated and left continuous functions on $\Rbar$ 
with limit $0$ at $-\infty$.  If
$f\in\acn$ or $f\in\arn$ for integer $n\in\NN$ then $\intinf fG$
exists if $G$ is an $n$-fold iterated integral of a function of 
bounded variation. If $f\in\acn$ or $\arn$ with primitive $F\in\balexc$ 
or $\balexr$ then
$\norm{f}_{a,n}=\norm{F}_\infty$ makes $\acn$ and $\arn$ into Banach spaces
that are isometrically isomorphic to $\Bc$ and $\Br$, respectively.
We also write $D^nF=F^{(n)}$.

Define $\Lnp=\{f\in\Sc'\mid f=D^nF \text{ for some } F\in L^p\}$
and denote by
$\Inq$ the functions $G\fn\RR\to\RR$ with 
$G(x)=\int_0^x\cdots\int_0^{x_{i+1}}\cdots
\int_0^{x_2}g(x_1)\,dx_1\cdots dx_i\cdots\,dx_n$ for some  $g\in L^q$.
Let $q$ be the
conjugate exponent of $p$. 
Then $\intinf fG=(-1)^n\intinf F(x)G^{(n)}(x)\,dx$ if $G\in\Inq$.
It follows that for each $0\leq m\leq n$,
$\intinf F^{(n)}G=(-1)^m\intinf F^{(n-m)}G^{(m)}$.

Results for the $\Lnp$ spaces are
analogous to those of $\Lpp =L\!^{(1),p}$.   Changing the lower limits in the
integrals defining a function $G\in\Inq$ changes $G$ by the addition
of a polynomial of degree at most $n-1$.  We define $\intinf fP=0$
for $f\in\Lnp$ and $P$ a polynomial of degree at most $n-1$.  
Theorem~\ref{theoremunique} holds without change and we define
$\norm{f}^{(n)}_p=\norm{F}_p$ where $f\in\Lnp$ and $F$ is its unique
primitive in $L^p$.  If $G\in\Inq$ is an $n$-fold iterated integral
of $g\in L^q$,  define $\norm{G}_{nI,q}=\norm{G^{(n)}}_q=\norm{g}_q$.  Then
$\Lnp$ is a Banach space with norm $\norm{\cdot}^{(n)}_p$ isometrically
isomorphic to $L^p$ and $\Inq$ is a Banach space with norm 
$\norm{\cdot}_{nI,q}$ isometrically
isomorphic to $L^q$.  The results of Theorem~\ref{theoremLp} hold
with these changes.  It then follows that the remaining  theorems
in Section~\ref{sectionLp} hold with minor changes. 

The convolution is
$f\ast G(x)=\intinf F^{(n)}(x-t)G(t)\,dt=
F\ast G^{(n)}(x)$ for $f\in L\!^{(n),p}$ and $G\in I^{n,q}$ and
$q$ conjugate to $p$.  
Analogous properties to Theorem~\ref{theoremconvinfty} can
now be seen to hold in $L\!^{(n),p}$.  
The second integral condition in (c) is replaced by
$\intinf |y^n\,h(y)|\,dy<\infty$.  Convolutions are defined in 
$L\!^{(n),1}\times L^p$ and in $L^1\times L\!^{(n),p}$ by replacing Definition~\ref{defnconvp} with
\begin{definition}\label{convn1}
Let $p,q,r\in[1,\infty)$ such that $1/p+1/q=1+1/r$.
Let $f\in L\!^{(n),p}$ with primitive $F\in L^p$ 
and let $g\in L^q$.  Let $(g_k)\subset L^q\cap I^{(n,q)}$
such that $\norm{g_k-g}_q\to 0$.  Define $f\ast g$ to be the
unique element in $L\!^{(n),r}$ such that 
$\norm{f\ast g-f\ast g_k}^{(n)}_r\to 0$.
Define $F\ast g^{(n)}$ to be the
unique element in $L\!^{(n),r}$ such that 
$\norm{F\ast g^{(n)}-F\ast g_k^{(n)}}^{(n)}_r\to 0$.
\end{definition}

Theorem~\ref{theoremconvLp} now holds with minor changes.

Using Definition~\ref{defnFourier}, the Fourier transform of 
$f\in L\!^{(n),1}$ with primitive $F\in L^1$ is $\fhat(s)=\intinf fe_s
=(is)^n\Fhat(s)$.  Then $\fhat$ is continuous with
$|\fhat(s)|= |s^n|\norm{\Fhat}_\infty\leq|s^n|\norm{F}_1=|s^n|\norm{f}_1^{(n)}$.
The Riemann--Lebesgue lemma becomes $\fhat(s)=o(s^n)$ as $|s|\to\infty$.
The rest of Theorem~\ref{theoremFourier} holds with minor changes.
In (i) the second condition on $g$ is that the function
$s\mapsto s^ng(s)$ is in $L^1$.

The space $L\!^{(n),2}$ is a Hilbert space with inner product
defined as per Definition~\ref{theoremHilbert}:
$(f,g)=(F,G)=\intinf F(x)G(x)\,dx$ where $f,g\in L\!^{(n),2}$ with
respective primitives $F,G\in L^2$.  The Fourier transform in
$L\!^{(n),2}$ is defined similarly and the Parseval equality
continues to hold.

Finally, we have the connection between $L\!^{(n),1}$ and $\alexc^m$.
\begin{prop}
Let $n\in\NN$. (a) $L\!^{(n),1}$ is a subspace of $\alexc^{n+1}$.  (b)  
It is not
closed.
(c) The norms $\norm{\cdot}^{(n)}_1$ and $\norm{\cdot}_{a,n+1}$ are
not equivalent.
\end{prop}

\begin{proof}
(a) If $f\in L\!^{(n),1}$ then there is $F\in L^1$ such that
$f=F^{(n)}$.  Let $G(x)=\int_{-\infty}^xF(t)\,dt$.  Then $G\in AC(\Rbar)
\subset\balexc$.  For almost all $x\in\RR$ we have $G'(x)=F(x)$ so
$G^{(n+1)}=F^{(n)}$ and $f\in\alexc^{n+1}$.
(b) It is not closed since $AC(\Rbar)$ is dense in $\balexc$.  See
\cite[Proposition~3.3]{talvilaconv}.
(c) Let $F_m(x)=\sin(mx)\chi_{(0,2\pi)}(x)$.  
Let $f_m=F_m^{(n)}\in L\!^{(n),1}$.
Then $\norm{f_m}^{(n)}_1=\norm{F_m}_1=\int_0^{2\pi}|\sin(mx)|\,dx=
4$.  And, $\norm{f_m}_{a,n+1}=\sup_{0\leq x\leq 2\pi}|\int_0^x\sin(mt)\,dt|
=\int_0^{\pi/m}\sin(mt)\,dt=2/m\to 0$ as $m\to\infty$.  The two norms
are then not equivalent.\end{proof}

For no $1<p<\infty$ is $L\!^{(n),p}$ a subset of any of the $\alexc^m$ spaces
since for each $1<p<\infty$ there is a function $F\in L^p$  for which
$\int_0^x F(t)\,dt$ is not bounded.

\section{Half plane Poisson integral}\label{sectionpoisson}
As an application we show the half plane Poisson integral can be
defined for distributions in $\Lnp$ and has essentially the same
behaviour with respect to $\norm{\cdot}^{(n)}_p$ as the Poisson
integral of $F\in L^p$ has with $\norm{\cdot}_p$.  The upper half plane is
$\Pi^+=\{(x,y)\in\RR^2\mid y>0\}$.  The Poisson kernel is
$\Phi_y(x)=(y/\pi)(x^2+y^2)^{-1}$.  If $F\in L^p$ define
$U_y(x)=\Phi_y\ast F(x)=
(y/\pi)\int_{-\infty}^\infty F(\xi)[(x-\xi)^2+y^2]^{-1}\,d\xi$.  The following
results are well known.
\begin{theorem}\label{theorempoissonLp}
Let $1\leq p<\infty$.  Let $F\in L^p$.  For each $(x,y)\in\Pi^+$ define 
$U_y(x)=\Phi_y\ast F(x)$. (a) $U_y$ is harmonic in $\Pi^+$.
(b) $\norm{U_y}_p\leq\norm{F}_p$
(c) $\norm{U_y-F}_p\to 0$ as $y\to0^+$.
\end{theorem}
For a proof see, for example, \cite{axler}.

We have the following analogue for distributions that are the $n$th
derivative of an $L^p$ function.
\begin{theorem}\label{theorempoissonLnp}
Let $1\leq p<\infty$.  Let $F\in L^p$. For each $(x,y)\in\Pi^+$ define 
$U_y(x)=\Phi_y\ast F(x)$.  Let $f=F^{(n)}\in\Lnp$. (a) For
each $(x,y)\in\Pi^+$ the integral $u_y(x)=\Phi_y\ast f(x)=
(y/\pi)\int_{-\infty}^\infty f(\xi)[(x-\xi)^2+y^2]^{-1}\,d\xi$ is well defined.
(b) $u_y$ is harmonic in $\Pi^+$.
(c) $\norm{u_y}^{(n)}_p\leq\norm{f}^{(n)}_p$
(d) $\norm{u_y-f}^{(n)}_p\to 0$ as $y\to0^+$.
\end{theorem}
\begin{proof}
Let $q$ be the conjugate of $p$.
(a) The Poisson kernel is real analytic in $\Pi^+$ so each derivative
is a continuous function.  Note that $\partial^m\Phi_y(x)/\partial x^m
=O(x^{-m-2})$ as $|x|\to\infty$.  Hence, each derivative of $\Phi_y$
is in $L^1$.  And, for each $y>0$ there is a polynomial, $p_y$, of
degree at most $n-1$ such that $\Phi_y+p_y\in I^{n,1}$. 
It follows from
Definition~\ref{convn1} that
\begin{eqnarray*}
u_y(x) & = &  \frac{(-1)^ny}{\pi}\intinf F(\xi)\frac{\partial^n\Phi_y(\xi-x)}{\partial \xi^n}
\,d\xi
  = \frac{y}{\pi} \intinf F(\xi)\frac{\partial^n\Phi_y(\xi-x)}{\partial x^n}
\,d\xi\\
 & = &  \frac{\partial^n U_y(x)}{\partial x^n}.
\end{eqnarray*}
The growth estimate $\partial^n\Phi_y(x)/\partial x^n
=O(x^{-m-2})$ as $|x|\to\infty$ and dominated convergence allows the
derivative to be moved 
outside the integral.  (b) Since the Laplacian commutes
with each derivative, we now see that $u_y$ is harmonic in $\Pi^+$.
(c) From (b) and Theorem~\ref{theorempoissonLp},
$\norm{u_y}^{(n)}_p=\norm{U_y}_p\leq\norm{F}_p=\norm{f}^{(n)}_p$.
(d) $\norm{u_y-f}^{(n)}_p=\norm{U_y-F}_p\to 0$ as $y\to0^+$.\end{proof}

The Poisson integral of distributions has been considered by
other authors in the spaces ${\mathcal D}'_{L^p}(\RR)$ (end of
Section~\ref{sectionnotation}) and weighted
versions of these spaces.  See \cite{alvarezguzman} and
\cite{cichockakierat} for results and
further references.  However, the boundary values are
then taken on in a weak sense, whereas here we have boundary values
taken on in the norm $\norm{\cdot}^{(n)}_p$.  A deeper study would
involve questions of uniqueness.  The condition on $F\in L^1$ in
Theorem~\ref{theorempoissonLnp} can be weakened to existence
of the integral $\intinf |F(x)|(x^2+1)^{-1}\,dx$.

\section{Integration in $\RR^n$}\label{sectionRn}
We briefly outline a method of extending the $\Lpp$
integrals to $\RR^n$.

Impose a Cartesian coordinate system on $\RR^n$ and write $x=(x_1,\dots, x_n)$.  
Let $1\leq p<\infty$
and let $q$ be its conjugate.
Let $F\in L^p(\RR^n)$.  Let $D=\partial^n/\partial x_1\partial x_2\cdots
\partial x_n$.  Define $\Lpp(\RR^n)=\{f\in\Sc'(\RR^n)\mid f=DF
\text{ for some } F\in L^p(\RR^n)\}$.
For the distributional derivative write $DF(x)=F_{12\cdots n}(x)$.
Let $G(x)=\int_0^{x_n}\cdots\int_0^{x_1}g(y_1)\,dy_1\cdots dy_n$ for
some $g\in L^q(\RR^n)$.  If $f=DF\in\Lpp(\RR^n)$ for unique $F\in L^p(\RR^n)$,
the integral is defined
with iterated Lebesgue integrals as
$\int_{\RR^n} f G=(-1)^n\intinf\cdots\intinf 
F(x)G_{12\cdots n}(x)\,dx_1\cdots dx_n$.  
The norm is $\norm{f}'_p=\norm{F}_p$ and $\Lpp(\RR^n)$ is isometrically
isomorphic to $L^p(\RR^n)$.  Most of the results of the
previous sections can now be extended to this integral.  This method
originates in work on integrals with continuous primitives by
Ang, Schmitt, Vy
\cite{ang} and Mikusi\'nksi, Ostaszewski \cite{pmikusinski}.
Similarly, if $\alpha_i\in\NN_0$ for $1\leq i\leq m$ and $\alpha
=(\alpha_1, \alpha_2, \ldots, \alpha_m)$ is a multi-index then define
the differential operator $\delta=\partial^{\alpha_1}/\partial x_1\cdots
\partial^{\alpha_m}/\partial x_m$.  Let $g\in L^q(\RR^n)$ and $G$ be an
iterated integral of $g$, integrated $\alpha_i$ times in the $i$th 
coordinate.  For $F\in L^p(\RR^n)$ and $f\in{\mathcal D}'(\RR^n)$ defined
by $f=\delta F$ the integral is 
$\int_{\RR^n} f G=(-1)^{\abs{\alpha}}\intinf\cdots\intinf 
F(x)\delta G(x)\,dx$.


\begin{thebibliography}{99}
\bibitem{alvarezguzman}
ALVAREZ,~J.---GUZM\'{A}N-PARTIDA,~M.---P\'{E}REZ-ESTEVA,~S.: 
\textit{Harmonic extensions of distributions}, Math. Nachr. \textbf{280} (2007),
1443--1466.
\bibitem{ang}
ANG,~D.~D.---SCHMITT,~K.---VY,~L.~K.:	
\textit{A multidimensional analogue of the Denjoy--Perron--Henstock--Kurzweil
integral}, 
Bull. Belg. Math. Soc. Simon Stevin \textbf{4} (1997), 355--371.
\bibitem{axler}
AXLER,~S.---BOURDON,~P.---RAMEY,~W.:
\textit{Harmonic function theory},
Springer--Verlag, New York, 2001.
\bibitem{barros-neto}
BARROS-NETO,~J.:
\textit{An introduction to the theory of distributions},
Marcel Dekker, New York, 1973.
\bibitem{bartleelements}
BARTLE, R.~G.:
\textit{Elements of integration}, Wiley, New York, 1966.
\bibitem{celidze}
\v{C}ELIDZE,~V.~G.---D\v{Z}VAR\v{S}E\v{I}\v{S}VILI,~A.~G.:
\textit{The theory of the Denjoy integral and some applications}
(P.~S.~Bullen, trans.),
World Scientific, Singapore, 1989.
\bibitem{cichockakierat}
CICHOCKA,~A.---KIERAT,~W.:
\textit{An application of the Wiener functions to the Dirichlet problem of the 
Laplace equation}, Integral Transform. Spec. Funct. \textbf{7} (1998), 13--20.
\bibitem{clarkson}
CLARKSON,~J.~A.:
\textit{Uniformly Convex Spaces},
Trans. Amer. Math. Soc.
\textbf{40} (1936), 396--414.
\bibitem{folland}
FOLLAND,~G.~B.:
\textit{Real analysis}, Wiley, New York, 1999.
\bibitem{friedlanderjoshi}
FRIEDLANDER,~F.~G.---JOSHI,~M.:
\textit{Introduction to the theory of distributions},
Cambridge University Press, Cambridge, 1999.
\bibitem{leader}
LEADER,~S.:
\textit{The Kurzweil--Henstock integral and its differentials},
Marcel Dekker, New York, 2001.
\bibitem{liebloss}
LIEB,~E.~H.---LOSS,~M.:
\textit{Analysis}, American Mathematical
Society, Providence, 2001.
\bibitem{pmikusinski}
MIKUSI\'{N}SKI,~P.---OSTASZEWSKI,~K.:
\textit{Embedding Henstock integrable
functions into the space of Schwartz distributions},
Real Anal. Exchange \textbf{14} (1988-89), 24--29.
\bibitem{rudinreal}
RUDIN,~W.:
\textit{Real and complex analysis}, McGraw--Hill, New York, 1987.
\bibitem{schwartz}
SCHWARTZ,~L.:
\textit{Th\`eorie des distributions},
Hermann, Paris, 1966.
\bibitem{talviladenjoy}
TALVILA,~E.:
\textit{The distributional Denjoy  integral}, Real Anal. Exchange
\textbf{33} (2008), 51--82.
\bibitem{talvilaconv}
TALVILA,~E.:
\textit{Convolutions with the continuous primitive integral},
Abstr. Appl. Anal. \textbf{2009} (2009), Art. ID 307404.
\bibitem{talvilaregulated}
TALVILA,~E.:
\textit{The regulated primitive integral}, Illinois J. Math.
\textbf{53} (2009), 1187--1219.
\bibitem{talvilaacrn}
TALVILA,~E.:
\textit{Integrals and Banach spaces for finite order distributions},
Czechoslovak Math. J. \textbf{62} (2012), 77--104.
\bibitem{yosida}
YOSIDA,~K.:
\textit{Functional analysis}, Springer--Verlag, Berlin, 1980.
\bibitem{zemanian}
ZEMANIAN,~A.~H.:
\textit{Distribution theory and transform analysis},
Dover, New York, 1987.
\end{thebibliography}
\end{document}